\documentclass[11pt,a4paper]{article}
\usepackage[utf8]{inputenc}
\usepackage[english]{babel}
\usepackage{amsmath}
\usepackage{amsfonts}
\usepackage{amssymb}
\usepackage{amsthm}
\usepackage{graphicx}
\usepackage{epsfig}
\usepackage{cite}
\usepackage[pagewise]{lineno} 
\usepackage{dsfont}
\usepackage[all]{xy}
\usepackage{indentfirst} 
\usepackage{url}

\newtheorem{theorem}{Theorem}[section]

\newtheorem{proposition}[theorem]{Proposition}
\newtheorem{lemma}[theorem]{Lemma}
\newtheorem{corollary}[theorem]{Corollary}
\newtheorem{remark}[theorem]{Remark}

\newtheorem{definition}{Definition}

\numberwithin{equation}{section}

\newcommand{\Sp}{\mathbb{S}}
\newcommand{\Hy}{\mathbb{H}}
\newcommand{\R}{\mathbb{R}}
\newcommand{\E}{\mathbb{E}}
\newcommand{\pl}{\langle}
\newcommand{\pr}{\rangle}
\newcommand{\Qt}{\mathbb{Q}}
\newcommand{\Q}{\mathbb{Q}_{\epsilon}^n}
\newcommand{\Qq}{\mathbb{Q}_{\epsilon}^{n+1}}
\newcommand{\dt}{\partial_t}
\newcommand{\dd}{\mathrm{d}}
\newcommand{\tr}{\mathrm{\,trace}\,}
\newcommand{\di}{\mathrm{div}\,}
\newcommand{\grad}{\mathrm{\,grad\, }}
\newcommand{\A}{\mathcal{A}}
\newcommand{\tf}{\tilde f}
\newcommand{\spa}{\mathrm{span}}
\newcommand{\Sn}{\Sp^n\times\R}
\newcommand{\Qn}{\Q\times\R}
\newcommand{\Qm}{\Qq\times\R}
\newcommand{\Hn}{\Hy^n\times\R}

\newcommand{\Qf}{\Qt^4_\epsilon\times\R}
\newcommand{\Id}{\text{Id}}
\newcommand{\diag}{\mathrm{diag}}

\newcommand{\Addresses}{{
\bigskip 
\footnotesize
Fernando Manfio -- Universidade de S\~ao Paulo, Brazil \\
\textit{E-mail address:} \texttt{manfio@icmc.usp.br}
\medskip\\
Nurettin Cenk Turgay -- Istanbul Technical University, Turkey \\
\textit{E-mail address:} \texttt{turgayn@itu.edu.tr}
\medskip\\
Abhitosh Upadhyay -- Harish Chandra Research Institute, India \\
\textit{E-mail address:} \texttt{abhi.basti.ipu@gmail.com, abhitoshupadhyay@hri.res.in}
}}

\title{Biconservative submanifolds in $\Sn$ and $\Hn$}
\date{}
\author{F. Manfio, N. C. Turgay and A. Upadhyay}

\begin{document}

\maketitle

\begin{abstract}
In this paper we study biconservative submanifolds in $\Sn$ and
$\Hn$ with parallel mean curvature vector field and co-dimension $2$.
We obtain some necessary and sufficient conditions for such 
submanifolds to be conservative. In particular, we obtain a complete
classification of $3$-dimensional biconservative submanifolds in 
$\Sp^4\times\R$ and $\Hy^4\times\R$ with nonzero parallel mean 
curvature vector field. We also get some results for biharmonic 
submanifolds in $\Sn$ and $\Hn$.
\end{abstract}

 \vspace{2ex}


\section{Introduction}

Roughly speaking, {\em biconservative} submanifolds arise as the 
vanishing of the stress-energy tensor associated to the variational
problem of biharmonic submanifolds. More precisely, an isometric
immersion $f:M\to N$ between two Riemannian manifolds is
biconservative if the tangent component of its bitension field
is identically zero (see Section \ref{sec:basic}).

Simplest examples of biconservative hypersurfaces in space forms
are those that have constant mean curvature. In this case, the 
condition of biconservative becomes $2A(\grad H)+H\grad H=0$,
where $A$ is the shape operator and $H$ is the mean curvature
function of the hypersurface. The case of surfaces in $\R^3$ was
considered by Hasanis-Vlachos \cite{HV}, and surfaces in $\Sp^3$
and $\Hy^3$ was studied by Caddeo-Montaldo-Oniciuc-Piu \cite{CMOP}.
In the Euclidean space $\R^3$, these surfaces are rotational.
Recent results in the study of biconservative submanifolds were obtained, for example, in \cite{Fu1, Fu2, FT, MOR,
MOR2, Tu, UT}.

Apart from space forms, however, there are few Riemannian manifolds
for which biconservative submanifolds are classified. Recently, this
was considered for surfaces with parallel mean curvature vector field
in $\Sp^n\times\R$ and $\Hy^n\times\R$ in \cite{FOP}, where they
found explicit parametrizations for such submanifolds.

In this paper, we give a complete classification of biconservative
submanifolds in $\Qf$ with nonzero parallel mean curvature vector
field and co-dimension $2$. This extends the one obtained in 
\cite{FOP}. To state our result, let $\Q$ denote 
either the unit sphere $\Sp^n$ or the hyperbolic space 
$\Hy^n$, according as $\epsilon=1$ or $\epsilon=-1$, respectively. 
Given an isometric immersion $f:M^m\to\Qn$, let $\dt$ be a unit 
vector field tangent to the second factor. Then, a tangent vector 
field $T$ on $M^m$ and a normal vector field $\eta$ along $f$ are 
defined by
\begin{eqnarray}\label{eq:dt}
\dt=f_\ast T+\eta.
\end{eqnarray}
Consider now an oriented minimal surface $\phi:M^2\to\Qt^2_a\times\R$
such that the vector field $T$ defined by \eqref{eq:dt} is 
nowhere vanishing, where $a\neq 0$ and $|a|<1$. Let $b>0$ be a
real number such that $a^2+b^2=1$. Let now 
\[
f:M^3=M^2\times I\to\Qf
\]
be given by
\begin{eqnarray}\label{eq:localf}
f(p,s)=\left(b\cos\frac{s}{b},b\sin\frac{s}{b},\phi(p)\right).
\end{eqnarray}

\begin{theorem}\label{theo:main}
The map $f$ defines, at regular points, an isometric immersion with 
$\pl H,\eta\pr=0$, where $H$ is the mean curvature vector field of $f$.
Moreover, $f$ is a biconservative isometric immersion
with parallel mean curvature vector field if and only if $\phi$ is a vertical
cylinder. Conversely, any biconservative isometric immersion
$f:M^3\to\Qf$ with nonzero parallel mean curvature vector field,
such that the vector field $T$ defined by \eqref{eq:dt} is nowhere 
vanishing, is locally given in this way.
\end{theorem}

In particular, we prove (see Corollary \ref{cor:main1}) that the 
submanifolds of Theorem \ref{theo:main} belong to a special class,
which consists of isometric immersions $f:M^m\to\Qn$ with the
property that the vector field $T$ is an eigenvector of all shape 
operators of $f$.

\vspace{.2cm}

The paper is organized as follows. In Section \ref{sec:basic}, we
recall some properties of biharmonic maps and we give a more 
precise statement of biconservative submanifolds. The basics
of submanifols theory in product space is discussed in Section
\ref{sec:prod}. In particular, we recall with details the class $\A$.
In Section \ref{sec:codim2} we show some general results 
about $n$-dimensional biconservative submanifolds in $\Qn$. 
In particular, we obtain a necessary and sufficient condition for a 
biconservative submanifold with parallel mean curvature vector 
to be biharmonic. Finally, Section \ref{sec:main} contains the 
arguments necessary to prove the above main Theorem.

\section{Preliminaries}\label{sec:basic}

Given a smooth map $f:M\to N$ between two Riemannian manifolds,
the energy density of $f$ is the smooth function $e(f):M\to\R$
defined by
\[
e(f)=\frac{1}{2}\|\dd f\|^2,
\]
where $\|\dd f\|^2$ denotes de Hilbert-Schmidt norm of $\dd f$. The 
total energy of $f$, denoted by $E(f)$, is given by integrating the 
energy density over $M$,
\[
E(f)=\frac{1}{2}\int_M\|\dd f\|^2\dd M.
\]
The map $f$ is called {\em harmonic} if it is a critical point of the energy 
functional $E$. Equivalently, $f$ is harmonic if it satisfies the Euler-Lagrange
equation $\tau(f)=0$, where 
\[
\tau(f)=\tr(\nabla\dd f)
\]
is known as the tension field of map $f$. When $f:M^m\to N^n$ is an
isometric immersion with mean curvature vector field $H$, we have
$\tau(f)=mH$. Therefore the immersion $f$ is a harmonic map
if and only if $M$ is a minimal submanifold of $N$.

A natural generalization of harmonic maps are the {\em biharmonic}
maps, which are critical points of the bienergy functional
\[
E_2(f)=\frac{1}{2}\int_M\|\tau(f)\|^2\dd M.
\]
This generalization, initially suggested by Eells-Sampson \cite{ES},
was studied by Jiang \cite{Ji2}, where he derived the corresponding
Euler-Lagrange equation
\[
\tau_2(f) = J(\tau(f)) = 0,
\]
where $J(\tau(f))=\Delta\tau(f)-\tr\tilde R(\dd f,\tau(f))\dd f$ is the 
Jacobi operator of $f$.

When $f:M^n\to N^n$ is an isometric immersion, we get
\[
\tau_2(f)=m(\Delta H-\tilde R(\dd f,H)\dd f).
\]
Thus a minimal isometric immersion in the Euclidean space is trivially
biharmonic. Concerning biharmonic submanifolds in the Euclidean space,
one of the main problem is the following known Chen's conjecture 
\cite{Ch1}: {\em Any biharmonic submanifold in the Euclidean space 
is minimal}.

\vspace{.2cm}

The {\em stress-energy} tensor, described by Hilbert \cite{Hi}, is a 
symmetric $2$-covariant tensor $S$ associated to a variational 
problem that is conservative at the critical points. Such tensor was 
employed by Baird-Eells \cite{BE} in the study of harmonic maps.
In this context, it is given by
\[
S=\frac{1}{2}\|\dd f\|^2\pl,\pr_M-f^\ast\pl,\pr_N,
\]
and it satisfies 
\[
\di S=-\pl\tau(f),\dd f\pr.
\]
Therefore, $\di S=0$ when $f$ is harmonic. 

\vspace{.2cm}

In the context of biharmonic
maps, Jiang \cite{Ji} obtained the stress-energy tensor $S_2$
given by
\begin{eqnarray*}
S_2(X,Y) &=& \frac{1}{2}\|\tau(f)\|^2\pl X,Y\pr + 
\pl\dd f,\nabla\tau(f)\pr\pl X,Y\pr \\
&& -\pl X(f),\nabla_Y\tau(f)\pr - \pl Y(f),\nabla_X\tau(f)\pr,
\end{eqnarray*}
which satisfies
\[
\di S_2=-\pl\tau_2(f),\dd f\pr.
\]
In the case of $f:M^m\to N^n$ to be an isometric immersion, it follows
that $\di S=0$, since $\tau(f)$ is normal to $f$. However, we have
\[
\di S_2=-\tau_2(f)^T,
\]
and thus $\di S_2$ does not always vanish.

\begin{definition}
{\em 
An isometric immersion $f:M^m\to N^n$ is called {\em biconservative} if its
stress-energy tensor $S_2$ is conservative, i.e., $\tau_2(f)^T=0$.}
\end{definition}

The following splitting result of the bitension field, with respect to
its normal and tangent components, is well known (see, for example
\cite{FOP, MOR, MOR2}).

\begin{proposition}\label{prop:splitting}
Let $f:M^m\to N^n$ be an isometric immersion between two Riemannian
manifolds. Then $f$ is biharmonic if and only if the tangent and normal
components of $\tau_2(f)$ vanish, i.e.,
\begin{eqnarray}\label{eq:propbihar1}
m\grad\|H\|^2 + 4\tr A_{\nabla^\perp_{(\cdot)}H}(\cdot) + 
4\tr\big(\tilde R(\cdot,H)\cdot\big)^T=0
\end{eqnarray}
and
\begin{eqnarray}\label{eq:propbihar2}
\tr\alpha_f(A_H(\cdot),H) - \Delta^\perp H + 
2\tr\big(\tilde R(\cdot,H)\cdot\big)^\perp=0,
\end{eqnarray}
where $\tilde R$ denotes the curvature tensor of $N$.
\end{proposition}

\section{Basic facts about submanifolds in $\Qn$}\label{sec:prod}

In order to study submanifolds $f:M^m\to\Qn$, our approach is to 
regard $f$ as an isometric immersion into $\E^{n+2}$, where 
$\E^{n+2}$ denote either Euclidean space or Lorentzian space
$(n+2)$-dimensional, according as $\epsilon=1$ or $\epsilon=-1$, 
respectively. Then we consider the canonical inclusion
\[
i:\Qn\to\E^{n+2}
\]
and study the composition $\tf=i\circ f$. Notice that the vector field 
$T$ is the gradient of the height function $h=\pl\tilde f,i_\ast\dt\pr$.

Using that $\dt$ is a parallel vector field in $\Qn$ we obtain, by 
differentiating \eqref{eq:dt}, that
\begin{eqnarray}\label{eq:vectorT}
\nabla_XT=A_\eta X
\end{eqnarray}
and 
\begin{eqnarray}\label{eq:vectorEta}
\alpha_f(X,T)=-\nabla^\perp_X\eta,
\end{eqnarray}
for all $X\in TM$, where $\alpha_f$ denotes the second fundamental
form of $f$ and $A_\xi$ stands for the shape operator of $f$ with respect
to $\xi\in TM^\perp$, given by
\[
\pl A_\xi X,Y\pr=\pl\alpha_f(X,Y),\eta\pr \
\text{for all} \ X,Y\in TM.
\]

The Gauss, Codazzi and Ricci equations for $f$ are, respectively 
\begin{eqnarray}
\begin{aligned}
R(X,Y)Z &= A_{\alpha(Y,Z)}X-A_{\alpha(X,Z)}Y+\epsilon\big(X\wedge Y \\
&+ \pl X,T\pr Y\wedge T- \pl Y,T\pr X\wedge T\big)Z,
\end{aligned}
\end{eqnarray}
\begin{eqnarray}\label{eq:codazzi}
\left(\nabla^\perp_X\alpha\right)(Y,Z)-\left(\nabla^\perp_Y\alpha\right)(X,Z)
=\epsilon\pl(X\wedge Y)T,Z\pr\eta
\end{eqnarray}
and
\begin{eqnarray}
R^\perp(X,Y)\xi=\alpha(X,A_\xi Y)-\alpha(A_\xi X,Y),
\end{eqnarray}
for all $X,Y,Z\in TM$ and $\xi\in TM^\perp$ (cf. \cite{LTV} for more details).

\vspace{.2cm}

In the case of hypersurfaces $f:M^n\to\Qn$, the vector field $\eta$
given in \eqref{eq:dt} can be written as 
\begin{eqnarray}\label{eq:etahyp}
\eta=\nu N,
\end{eqnarray}
where $N$ is a unit normal vector field along $f$ and $\nu$ is a 
smooth function on $M$. Thus the equations \eqref{eq:vectorT} and \eqref{eq:vectorEta} become
\[
\nabla_XT=\nu AX \quad\text{and}\quad X(\nu)=-\pl AX,T\pr,
\]
for all $X\in TM$, where $A$ stands for the shape operator of $f$ with
respect to $N$.

\subsection{The class $\A$}\label{sec:classA}

We will denote by $\A$ the class of isometric immersions $f:M^m\to\Qn$
with the property that $T$ is an eigenvector of all shape operators of $f$.
This class was introduced in \cite{To}, where a complete description was
given for hypersurfaces, and extended to submanifolds of $\Qn$ in \cite{MT}.
Trivial examples are the slices $\Q\times\{t\}$, corresponding to the case
in which $T$ vanishes identically, and the vertical cylinders $N^{m-1}\times\R$, where $N^{m-1}$ is a submanifold of $\Q$, which correspond to the case
in which the normal vector field $\eta$ vanishes identically.

Following the notation of \cite{MT}, let us recall a way of construct 
more examples of submanifolds in this class. Let $g:N^{m-1}\to\Q$
be an isometric immersion and suppose that there exists an orthonormal
set of parallel normal vector fields $\xi_1,\ldots,\xi_k$ along $g$. Thus
the vector subbundle $E$ with rank $k$ of $TN^\perp$, spanned by
$\xi_1,\ldots,\xi_k$, is parallel and flat. Let us denote by $j:\Q\to\Qn$ and
$i:\Qn\to\E^{n+2}$ the canonical inclusions, and let $l=i\circ j$. Set
\[
\tilde\xi_0=l\circ g, \quad \tilde\xi_i=l_\ast\xi_i, \ 1\leq i\leq k, 
\quad\text{and}\quad \tilde\xi_{k+1}=i_\ast\dt.
\]
Then the vector subbundle $\tilde E$ of $TN^\perp_{\tilde g}$, where
$\tilde g=l\circ g$, spanned by $\tilde\xi_0,\ldots,\tilde\xi_{k+1}$, is
parallel and flat, and we can define a vector bundle isometry
\[
\phi:N^{m-1}\times\E^{k+2}\to\tilde E
\]
by
\[
\phi(x,y)=\sum_{i=0}^{k+1}y_i\tilde\xi_i(x),
\]
for all $x\in N^{m-1}$ and for all $y=(y_0,\ldots,y_{k+1})\in\E^{k+2}$.
Using this isometry, we define a map $f:N^{m-1}\times I\to\Qn$
by
\begin{eqnarray}\label{eq:mapf}
\tilde f(x,t)=(i\circ f)(x,t)=\phi(x,\alpha(t)),
\end{eqnarray}
where $\alpha:I\to\mathbb Q^k\times\R$ is a regular curve 
with $\sum_{i=0}^k\alpha_i^2=1$ and $\alpha_{k+1}'\neq0$.

The main result concerning the map $f$ given in \eqref{eq:mapf} is
that, at regular points, $f$ is an isometric immersion in class $\A$.
Conversely, given any isometric immersion $f:M^m\to\Qn$
in class $\A$, with $m\geq2$, $f$ is locally given in this way 
(cf. \cite[Theorem 2]{MT}). The map $\tf$ is a partial tube over
$\tilde g$ with type fiber $\alpha$ in the sense of \cite{CW}.
Geometrically, the submanifold $M^m=N^{m-1}\times I$ is obtained
by the parallel transport of $\alpha$ in a product submanifold 
$\mathbb Q^k\times\R$ of a fixed normal space of $\tilde g$ with respect
to its normal connection.

We point out that, in the case of hypersurfaces, $f$ is in class $\A$
if and only if the vector field $T$ in \eqref{eq:dt} is nowhere vanishing
and $\tf$ has flat normal bundle (cf. \cite[Proposition 4]{To}).
Some important classes of hypersurfaces of $\Qn$ that
are included in class $\A$ are hypersurfaces with constant sectional
curvature \cite{MaT}, rotational hypersurfaces \cite{DFV} and constant
angle hypersurfaces \cite{To}. For submanifolds of higher codimension,
we have that $f$ is in class $\A$ and it has flat normal bundle if and
only if the vector field $T$ in \eqref{eq:dt} is nowhere vanishing and
$\tf$ has flat normal bundle \cite[Corollary 3]{MT}.

\section{Biconservative submanifolds in $\Qn$}\label{sec:codim2}

Let $f:M^m\to\Qn$ be an isometric immersion with nonzero parallel 
mean curvature vector field $H$. It follows from \eqref{eq:propbihar1}
and from the expression of the curvature tensor of $\Qn$ that $f$ is 
biconservative if and only if
\begin{eqnarray}\label{eq:biconprod}
\epsilon\pl H,\eta\pr T=0,
\end{eqnarray}
where $\eta$ and $T$ denote the vector fields given in \eqref{eq:dt}. 
Without loss of generality, we may assume that $T$ and $\eta$ are 
nowhere vanishing. Therefore, it follows from \eqref{eq:biconprod} 
that $H$ is orthogonal to $\dt$ and, thus
\begin{eqnarray}\label{eq:biconprod2}
X\pl H,\dt\pr=0,
\end{eqnarray}
for all $X\in TM$. As $\nabla^\perp_XH=0$ and $\tilde\nabla_X\dt=0$,
it follows from \eqref{eq:biconprod2} that
\[
\pl A_HT,X\pr=0,
\]
for all $X\in TM$, which implies that 
\begin{eqnarray}\label{eq:biconprod3}
A_HT=0.
\end{eqnarray}

On the other hand, since $H$ is parallel it follows from the Ricci equation
that $[A_H,A_\xi]=0$ for every $\xi\in TM^\perp_f$. In particular, we have
$[A_H,A_\eta]=0$. Equivalently, the eigenspaces associated to $A_H$ 
are invariant by $A_\eta$. In particular, if we denote by
\begin{eqnarray}\label{eq:biconprod4}
E_0(H)=\{X\in TM:A_HX=0\},
\end{eqnarray}
we conclude that 
\begin{eqnarray}\label{eq:biconprod5}
A_\eta T\in E_0(H).
\end{eqnarray}

\begin{remark}\label{rem:codim1}
{\em
In the case of biconservative hypersurfaces with nonzero parallel
mean curvature vector, the equation \eqref{eq:biconprod} can be 
written as
\[
h\nu T=0,
\]
where $\nu$ is the function given in \eqref{eq:etahyp} and $h$ is
the smooth function such that $H=hN$. Thus, as $h\neq0$, a 
biconservative hypersurface $f:M^n\to\Qn$ with nonzero parallel
mean curvature vector is either a slice $\Q\times\{t\}$ or an open
subset of a Riemannian product $N^{n-1}\times\R$, where $N^{n-1}$
is a hypersurface of $\Q$ with nonzero parallel mean curvature 
vector field.}
\end{remark}

Thus, by virtue of Remark \ref{rem:codim1}, we will consider 
biconservative submanifolds with codimension greater than one.

\subsection{Biconservative submanifolds of co-dimension $2$}

Let us consider now the case of co-dimension $2$, that is, a biconservative
isometric immersion $f:M^n\to\Qm$ with nonzero parallel mean curvature
vector field $H$. Let us consider the unit normal vector fields
\begin{eqnarray}\label{eq:xi1xi2}
\xi_1=H/\|H\| \quad\text{and}\quad \xi_2=\eta/\|\eta\|.
\end{eqnarray}
It follows from \eqref{eq:biconprod} that $\{\xi_1,\xi_2\}$ is an orthonormal
normal frame of $f$. Moreover, as $\nabla^\perp\xi_1=0$ and $f$ has
co-dimension $2$, we also have $\nabla^\perp\xi_2=0$.

Suppose first that the eigenspace $E_0(H)$ given in \eqref{eq:biconprod4}
is one-dimensional, that is, $E_0(H)=\spa\{T\}$. This implies that
$A_\eta T=\lambda T$ for some smooth function $\lambda$. Thus,
it follows from \eqref{eq:vectorEta} that
\[
-\nabla^\perp_X\eta = \frac{\lambda}{\|\eta\|^2}\pl T,X\pr\eta.
\]
In particular, we have $\nabla^\perp_X\eta=0$ for every $X\in\{T\}^\perp$
and, from \cite[Proposition 10]{MT}, we conclude that $f$ is in class $\A$.

\begin{remark}
If $E_0(H)$ is $n$-dimensional one has $A_{\xi_1}$ identically zero.
This implies that the mean curvature vector field $H$ of $f$ is a 
multiple of $\xi_2$, and this contradicts the fact that $H$ and 
$\eta$ are orthogonal, unless that $\eta$ is identically zero.
\end{remark}

From now on, let us assume that $\dim E_0(H)=k$, with $1<k<n$.

\begin{lemma}\label{lem:orthframe}
Let $f:M^n\to\Qm$ be a biconservative isometric immersion with
nonzero parallel mean curvature vector field. Then there exists a local
orthonormal frame $X_1,\ldots,X_n$ in $M^n$, with $X_1=T/\|T\|$,
such that:
\begin{itemize}
\item[(i)] The shape operators of $f$ with respect to $\xi_1$ and $\xi_2$,
given in \eqref{eq:xi1xi2}, have matrix representations given by
\begin{equation}\label{ShapeOpFORMS1}
A_{\xi_1}=
\left(\begin{array}{cc}
0 & 0 \\
0 & S_1
\end{array}\right)
\quad\mbox{ and }\quad
A_{\xi_2}=
\left(\begin{array}{cc}
S_2 & 0 \\
0 & B
\end{array}
\right),
\end{equation}
where $S_1$ and $B$ are diagonalized matrices and $S_2$ is a 
symmetric matrix such that $\tr S_1=const\neq0$, $\tr S_2+\tr B=0$
and
\begin{eqnarray}\label{eq:ShOpFo2}
\tr BS_1=0.
\end{eqnarray}
\item[(ii)] $\nabla_{X_i}X_j\in E_0(H)$, for all $1\leq i,j\leq k$.
\end{itemize}
\end{lemma}
\begin{proof}
Writing $X_1=T/\|T\|$, consider the local orthonormal frame 
$X_1,X_2,\ldots,X_n$ in $M^n$, where $X_2,\ldots,X_n$ are
eigenvectors of $A_{\xi_1}$ such that 
\[
E_0(H)=\spa\{X_1,X_2,\ldots,X_k\}.
\]
Thus we have the first equation of \eqref{ShapeOpFORMS1}. Moreover,
since $\xi_1$ is proportional to $H$ and $H$ has constant length, we have
$\tr A_{\xi_1}=\tr S_1=const\neq0$ and $\tr A_{\xi_2}=0$. On the other 
hand, by a simple computation, one can see that the Ricci equation 
$R^\perp(X_i,X_j)\xi_1=0$ takes the form
\begin{equation}\label{RicciEqGeneral}
\alpha_f(X_i,A_{\xi_1}X_j)-\alpha_f(A_{\xi_1}X_i,X_j)=0.
\end{equation}
For $1\leq i\leq k$ and $k+1\leq j\leq n$, the equation \eqref{RicciEqGeneral}
gives
\begin{eqnarray}\label{RicciEqGeneral2}
\alpha_f(X_i,X_j)=0.
\end{eqnarray}
Therefore, the matrix representation of $A_{\xi_2}$ takes the form given
in the second equation of \eqref{ShapeOpFORMS1}. Moreover, 
for $k+1\leq i\neq j\leq n$, \eqref{RicciEqGeneral} becomes
\[
(\lambda_i-\lambda_j)\alpha_f(X_i,X_j)=0.
\]
Now, if the distribution 
\[
\Gamma_i=\{X\in TM: A_{\xi_1}X=\lambda_iX\}
\]
has dimension $m_i>1$, then we have 
$A_{\xi_1}\vert_{\Gamma_i}=\lambda_i\Id$ and, by replacing indices 
if necessary, we may assume that
$\Gamma_i=\spa\{X_i,X_{i+1},\ldots,X_{i+m_i-1}\}$. Therefore, by redefining
$X_i,X_{i+1},\ldots,X_{i+m_i-1}$ properly, we may diagonalize 
$A_{\xi_2}\vert_{\Gamma_1}$. Since $A_{\xi_1}\vert_{\Gamma_1}$ is 
proportional to identity matrix it, no matter, remains diagonalized. Summing
up, we see that, by redefining $X_{k+1},X_{k+2},\ldots,X_n$ properly, 
one can diagonalize the matrix $B$. Then, we can write
\[
S_1=\diag(\lambda_{k+1},\lambda_{k+2},\ldots,\lambda_n)
\quad\text{and}\quad
B=\diag(\mu_{k+1},\mu_{k+2},\ldots,\mu_n)
\]
for some smooth functions $\lambda_i,\mu_i$, with $k+1\leq i\leq n$. In
order to obtain \eqref{eq:ShOpFo2}, we need to show that
\begin{eqnarray}\label{eq:ShOpFo3}
\sum_{i=k+1}^n\lambda_i\mu_i=0.
\end{eqnarray}
By a direct computation, it follows from the Codazzi equation that
\begin{eqnarray}\label{eq:Codlemma}
X_1(\lambda_i)+\pl\nabla_{X_i}X_1,X_i\pr\lambda_i=0,
\end{eqnarray}
for $k+1\leq i\leq n$. On the other hand, from \eqref{eq:dt} we have
\begin{eqnarray}\label{eq:dt2}
\dt=\cos\theta X_1+\sin\theta\xi_2
\end{eqnarray}
for a smooth function $\theta\neq\frac{\pi}{2}$. Since $\dt$ is parallel,
equation \eqref{eq:dt2} yields 
\begin{eqnarray}\label{eq:dt3}
0=\cos\theta\pl\nabla_{X_i}X_1,X_i\pr - \sin\theta\pl A_{\xi_2}X_i,X_i\pr.
\end{eqnarray}
Combining \eqref{eq:Codlemma} and \eqref{eq:dt3}, we get
\[
X_1(\lambda_i)=\tan\theta\lambda_i\mu_i.
\]
By summing this equation on $i$ and taking into account 
\[
\tr S_1=\sum_{i=k+1}^n\lambda_i=const,
\]
we get \eqref{eq:ShOpFo3}, which proves the assertion in (i).
Finally, for $1\leq i,l\leq k$ and $k+1\leq j\leq n$, we obtain from
Codazzi equation that
\[
\pl\tilde R(X_i,X_j)X_l,\xi_1\pr=0.
\]
Then, using \eqref{RicciEqGeneral2}, we obtain 
\[
\pl\nabla_{X_i}X_l,X_j\pr=0,
\]
for all $1\leq i,l\leq k$ and $k+1\leq j\leq n$, and this proves (ii).
\end{proof}

\begin{corollary}\label{cor:inv}
Let $f:M^n\to\Qm$ be a biconservative isometric immersion with
nonzero parallel mean curvature vector field. Then $E_0(H)$ is
an involutive distribution.
\end{corollary}
\begin{proof}
It is clear when $\dim E_0(H)=1$. If $\dim E_0(H)>1$, consider a local
orthonormal frame $X_1,\ldots,X_n$ in $M^n$ constructed in Lemma
\ref{lem:orthframe}. From condition (ii), we have  $[X,Y]\in E_0(H)$,
for all $X,Y\in E_0(H)$, which completes the proof.
\end{proof}

In the next result we obtain a necessary and sufficient condition for a
biconservative submanifold with parallel mean curvature vector to be 
biharmonic.

\begin{proposition}\label{PropBihEqForCodim2}
Let $f:M^n\to\Qm$ be a biconservative isometric immersion with nonzero
parallel mean curvature vector field. Then, $M$ is biharmonic if and only if
the equation
\begin{equation}\label{BihEqForCodim2}
\tr A_{\xi_1}^2+\|T\|^2=n
\end{equation}
is satisfied, where $\xi_1$ is the unit normal vector field given in 
\eqref{eq:xi1xi2}.
\end{proposition}
\begin{proof}
By Proposition \ref{prop:splitting}, $M$ is biharmonic if and only if the
equation \eqref{eq:propbihar2} is satisfied. Consider the local orthonormal
frame $\{X_1,\ldots,X_n\}$ given in Lemma \ref{lem:orthframe}. Since, the
mean curvature vector field $H$ is parallel and $\pl H,\eta\pr=0$, the
equation \eqref{eq:propbihar2} turns into \eqref{BihEqForCodim2} by virtue
of \eqref{eq:ShOpFo2}.
\end{proof}

\section{Biconservative submanifolds in $\Qf$}\label{sec:main}

In this section we prove Theorem \ref{theo:main} in two steps.
In the fist one, we prove that there is an explicit way to construct
$3$-dimensional biconservative submanifolds in $\Qf$ with
parallel mean curvature vector field. In the second step, we prove
that any $3$-dimensional biconservative submanifolds in $\Qf$,
with nonzero parallel mean curvature vector field, is locally given
as in the previous construction.

\subsection{Examples of biconservative submanifolds}

Here we prove the first part of Theorem \ref{theo:main}.

\begin{theorem}\label{theo:main1}
Let $\phi:M^2\to\Qt^2_a\times\R$ be an oriented minimal surface
such that the vector field $T_\phi$ defined by \eqref{eq:dt} is 
nowhere vanishing, where $a\neq 0$ and $|a|<1$. Let $b>0$ be a
real number such that $a^2+b^2=1$. Let now 
\[
f:M^3=M^2\times I\to\Qf
\]
be given by
\begin{eqnarray}\label{eq:localf}
f(p,s)=\left(b\cos\frac{s}{b},b\sin\frac{s}{b},\phi(p)\right).
\end{eqnarray}
Then the map $f$ defines, at regular points, an isometric immersion with 
$\pl H,\eta\pr=0$. Moreover, $f$ is a biconservative isometric immersion
with parallel mean curvature vector field if and only if $\phi$ is a vertical
cylinder.
\end{theorem}
\begin{proof}
Let $\{X_1,X_2,X_3\}$ be a local orthonormal tangent frame of $M^3$,
with $X_3=\partial_s$. By putting $Y_1=\pi_\ast X_1$ and 
$Y_2=\pi_\ast X_2$, where $\pi:M^3\to M^2$ denotes the canonical
projection, $\pi(p,s)=p$, we get that $\{Y_1,Y_2\}$ is a
local orthonormal tangent frame of $M^2$. If $N=(N_1,N_2,N_3,N_4)$ 
denotes the unit normal
vector field of $M^2$ in $\Qt^2_a\times\R$, then
\[
\xi_1=\left(-a\cos\frac{s}{b},-a\sin\frac{s}{b},\frac{b}{a}(\pi_1\circ\phi)\right)
\quad\text{and}\quad
\xi_2=(0,0,N)
\]
provides a local orthonormal normal frame of $f$ in $\Qf\subset\E^6$,
where $\pi_1:\Qt^2_a\times\R\to\Qt^2_a$ denotes the canonical projection.
Note that we have 
\[
\pl\xi_1,\dt\pr=0.
\]
In terms of the tangent frame $\{Y_1,Y_2\}$ of $M^2$, the shape 
operator $A_N$ is given by
\[
A_N=\left(
\begin{array}{cc}
a_{11}&a_{12}\\
a_{12}&-a_{11}
\end{array}
\right),
\]
for some smooth functions $a_{11}$ and $a_{12}$. By a direct
computation, one can see that the matrix representation of $A_{\xi_2}$,
with respect to $\{X_1,X_2,X_3\}$, take the form 
\begin{equation}\label{eq:matrixi2}
A_{\xi_2}=\left(
\begin{array}{ccc}
a_{11}&a_{12}&0\\
a_{12}&-a_{11}&0\\
0&0&0
\end{array}
\right).
\end{equation}
It follows from \eqref{eq:matrixi2} that $H=c\cdot\xi_1$, where 
$c=\pl H,\xi_1\pr$, which implies $\pl H,\eta\pr=0$. Moreover,
we have $T_\phi = \pi_\ast T_f$, since $\pl\partial_t,\partial_s\pr=0$.
Thus, as $T_\phi$ is nowhere vanishing, and therefore also $T_f$,
it is straightforward to verify that $H$ is parallel
if and only if $N_4=0$. It means that $\partial_t$ is orthogonal to $M^2$,
which implies that $\|T_\phi\|=1$. Thus, $M^2$ is a vertical cylinder
$M^2=\gamma\times\R$ over a geodesic curve $\gamma$ in $\Qt^2_a$.
\end{proof}

\begin{corollary}\label{cor:main1}
If $f:M^3\to\Qf$ is a biconservative submanifold with nonzero parallel
mean curvature vector field, locally given as in \eqref{eq:localf}, then
$f$ is an immersion in class $\A$.
\end{corollary}
\begin{proof}
As $f$ is locally given as in \eqref{eq:localf} it follows, in particular,
that $\phi$ is in class $\A$. Thus, the vector field $T_\phi$ 
associated to $\phi$, given in \eqref{eq:dt}, is a principal
direction of $\phi$. This implies that
\[
\pl A_\zeta T_\phi,Z\pr = 0,
\]
for all $\zeta\in TM^\perp_\phi$, where $Z$ is tangent to $\phi$ and
orthogonal to $T_\phi$. With the notations as in Theorem 
\ref{theo:main1}, and by considering 
\[
Y_1=\frac{T_\phi}{\|T_\phi\|} \quad\text{and}\quad
Y_2=\frac{Z}{\|Z\|},
\]
we have
\[
A_N=\left(
\begin{array}{cc}
a_{11} & 0 \\
0 & a_{22}
\end{array}
\right)
\quad\text{and}\quad
A_{\xi_2}=\left(
\begin{array}{ccc}
a_{11} & 0 & 0 \\
0 & a_{22} & 0 \\
0 & 0 & 0
\end{array}
\right).
\]
This shows that $T_f$ is an eigenvector of $A_{\xi_2}$, since 
$T_\phi = \pi_\ast T_f$.
\end{proof}

\subsection{Classification results in $\Qf$}

Finally, in this subsection, we prove the converse of Theorem
\ref{theo:main}. Here we will consider biconservative isometric immersion
$f:M^3\to\Qf$, with nonzero parallel mean curvature vector field $H$
such that $\dim E_0(H)=2$. Let us consider the local orthonormal 
frame $\{X_1,X_2,X_3\}$ given in Lemma
\ref{lem:orthframe}. Denoting by $\xi_1$ and $\xi_2$ as in \eqref{eq:xi1xi2},
we have
\begin{equation}\label{Subm3inS4RShapeOps1a}
A_{\xi_1}=\left(
\begin{array}{ccc}
0 & 0 & 0 \\
0 & 0 & 0 \\
0 & 0 & 3\|H\| \\
\end{array}
\right)
\end{equation}
and 
\begin{equation}\label{Subm3inS4RShapeOps1b}
A_{\xi_2}=\left(
\begin{array}{ccc}
a_{11} & a_{12} & 0 \\
a_{12} & a_{22} & 0 \\
0 & 0 & a_{33} \\
\end{array}
\right),
\end{equation}
for some smooth functions $a_{11}$, $a_{22}$ and $a_{33}$, with
$a_{11}+a_{22}+a_{33}=0$. Note that, from \eqref{eq:ShOpFo2}, we
have $a_{33}=0$ and thus, \eqref{Subm3inS4RShapeOps1b} becomes 
\eqref{eq:matrixi2}.

On the other hand, we can write the vector field $\partial_t$ as
\begin{eqnarray}\label{eq:delttheta}
\partial_t=\cos\theta X_1+\sin\theta\xi_2,
\end{eqnarray}
for a smooth function $\theta\neq\frac\pi2.$ Applying $X_3$ to 
\eqref{eq:delttheta}, we obtain 
\[
\nabla_{X_3}X_1=0.
\]
Moreover,
from the Codazzi equation, we obtain $\pl\tilde R(X_2,X_3)X_3,\xi_1\pr=0$,
that implies
\begin{eqnarray}\label{Eq1Fornablae3e1}
\pl\nabla_{X_3}X_2,X_3\pr=0.
\end{eqnarray}

By putting $\pl\nabla_{X_i}X_1,X_2\pr=\phi_i$, for $1\leq i\leq 2$,
we have the following:

\begin{lemma}\label{lem:LCivita}
In terms of the local orthonormal frame $\{X_1,X_2,X_3\}$ in $M^3$,
the Levi-Civita connection of $M^3$ is given by
\begin{eqnarray}\label{eq:LCivita}
\begin{array}{lll}
\nabla_{X_1}X_1=\phi_1X_2, & \nabla_{X_1}X_2=-\phi_1X_1, &
\nabla_{X_1}X_3=0, \\
\nabla_{X_2}X_1=\phi_2X_2, & \nabla_{X_2}X_2=-\phi_2X_1, &
\nabla_{X_2}X_3=0, \\
\nabla_{X_3}X_1=0, & \nabla_{X_3}X_2=0, & \nabla_{X_3}X_3=0. \\
\end{array}
\end{eqnarray}
\end{lemma}
\begin{proof}
A straightforward computation.
\end{proof}

\begin{lemma}\label{lem:system}
There exists a local coordinate system $(u_1,u_2,s)$ in $M^3$ such 
that $E_0(H)=\spa\{\partial_{u_1},\partial_{u_2}\}$, $X_3=\partial_s$
and $f$ decomposes as 
\begin{eqnarray}\label{eq:system4}
f(u_1,u_2,s)=\Gamma_1(s)+\Gamma_2(u_1,u_2),
\end{eqnarray}
for some smooth functions $\Gamma_1$ and $\Gamma_2$. Moreover,
$f(u_1,u_2,\cdot)$ are the integral curves of $X_3$ for any $(u_1,u_2)$
and $f(\cdot,\cdot,s)$ are the integral submanifolds of $E_0(H)$ for any
$s$.
\end{lemma}
\begin{proof}
By Corollary \ref{cor:inv}, the tangent bundle $TM$ decomposes
orthogonally as
\[
TM=E_0(H)\oplus(E_0(H))^\perp.
\]
Therefore, there exists a local coordinate
system $(u_1,u_2,s)$ in $M^3$ such that 
\[
E_0(H)=\spa\{\partial_{u_1},\partial_{u_2}\} 
\quad\text{and}\quad
(E_0(H))^\perp=\spa\{\partial_{s}\}
\]
(see \cite[p. 182]{KN}). Thus $X_3=E\partial_s$ for some smooth 
function $E$ on $M^3$. On the other hand, since 
$[\partial_{u_i},\partial_s]=0$, for $1\leq i\leq 2$, we have
$\nabla_{\partial_{u_i}}\partial_s=\nabla_{\partial_s}\partial_{u_i}$.
However, by considering \eqref{eq:LCivita}, one can see that
\begin{eqnarray}\label{eq:system}
\nabla_{\partial_{u_i}}\partial_s\in E_0(H)^\perp
\quad\text{and}\quad
\nabla_{\partial_s}\partial_{u_i}\in E_0(H),
\end{eqnarray}
for $1\leq i\leq 2$. By considering \eqref{eq:system} and \eqref{eq:LCivita},
and taking into account the fact that $\alpha_{\tilde f}(X,Y)=\alpha(X,Y)$,
whenever $X,Y$ are orthogonal tangent vector fields on $M^3$, we
obtain 
\begin{eqnarray}\label{eq:system2}
\hat\nabla_{\partial_{u_i}}\partial_s=\hat\nabla_{\partial_s}\partial_{u_i}=0
\end{eqnarray}
and
\begin{eqnarray}\label{eq:system3}
\partial_{u_i}(E)=0, 
\end{eqnarray}
for all $1\leq i\leq 2$. From \eqref{eq:system2}, we obtain 
\eqref{eq:system4} for some
smooth functions $\Gamma_1$ and $\Gamma_2$. Moreover, 
equation \eqref{eq:system3} implies that $E=E(s)$. Therefore, by
re-defining the parameter $s$ properly, we may assume that $E=1$,
which concludes the proof.
\end{proof}

\begin{proposition}\label{PROPPP2}
Let $f:M^3\to\Qf$ be a biconservative isometric immersion with
nonzero parallel mean curvature vector field $H$. Suppose that
$\dim E_0(H)=2$ and let $p\in M$. Then the following assertions
hold: 
\begin{enumerate}
\item[(i)] An integral submanifold $N$ of $E_0(H)$ through $p$ lies on a
$4$-plane $\Pi_1$ of $\E^6$ containing the factor $\partial_t$. Moreover,
$N$ is congruent to a minimal surface $\phi:M^2\to\Qt^2_a\times\R$. 
\item[(ii)] An integral curve of $X_3$ through $p$ is an open subset of a
circle of radius $b=\frac{1}{\sqrt{c^2+1}}$ contained on a $2$-plane 
$\Pi_2$ of $\E^6$, where $c=3\|H\|$.
\end{enumerate}
\end{proposition}
\begin{proof}
Let $N$ be an integral submanifold of $E_0(H)$ through $p$. Define 
vector fields $\zeta_1,\ldots,\zeta_6$ along $N$ by 
\[
\zeta_i=X_i\vert_N \quad\text{and}\quad \zeta_j=\xi_j\vert_N,
\]
for $1\leq i\leq3$ and $1\leq j\leq3$, where $\xi_3$ is the restriction of
the unit normal vector field of the immersion $f:M^3\to\Qf$ to $M^3$.
Note that $\zeta_1,\zeta_2$ span $TN$, while the vector fields
$\zeta_3,\ldots,\zeta_6$ span the normal bundle $TN^\perp$ in $\E^6$.
By taking into account the fact that $\alpha_{\tilde f}(X,Y)=\alpha_f(X,Y)$,
whenever $X,Y$ are orthogonal tangent vector fields on $M^3$, and 
considering \eqref{eq:matrixi2}, \eqref{Subm3inS4RShapeOps1a} and
Lemma \ref{lem:LCivita}, we get
\[
\hat\nabla_X\zeta_3=\hat\nabla_X\zeta_4=0,
\]
for all $X\in TM$, where $\hat\nabla$ is the Levi-Civita connection
of $\E^6$. This yields that $N$ lies on a $4$-plane $\Pi_1$ on which 
$\partial_t$ lies. Moreover, the unit normal vector field of $N$ in 
$\Pi_1\cap(\Qf)\cong\Qt^2_a\times\R$ is $\zeta_5$, and the
shape operator of $N$ along $\zeta_5$ becomes
\[
A_{\zeta_5}=\left(
\begin{array}{cc}
a_{11}&a_{12}\\
a_{12}&-a_{12}\\
\end{array}
\right),
\]
which shows that $N$ is congruent to a minimal surface in 
$\Qt^2_a\times\R$. This proves the assertion (i). In order to prove (ii),
let us consider an integral curve $\gamma$ of $X_3$ through $p$ and
define
\[
\zeta = \frac{c}{\sqrt{c^2+1}}\zeta_1-\frac{1}{\sqrt{c^2+1}}\zeta_3
\]
as a vector field along $\gamma$. Then we have
\[
\hat\nabla_{\gamma'}\gamma'=\sqrt{c^2+1}\zeta 
\quad\text{and}\quad
\hat\nabla_{\gamma'}\zeta=-\sqrt{c^2+1}\gamma'.
\]
Thus $\gamma$ is an open subset of a circle lying on the $2$-plane
$\Pi_2$ spanned by $\gamma'$ and $\zeta$. This proves (ii) and
concludes the proof.
\end{proof}

By summing up Lemma \ref{lem:system} and Proposition \ref{PROPPP2},
we get the converse of Theorem \ref{theo:main}, which can be stated
as follow.

\begin{theorem}\label{MainThm1}
Let $f:M^3\to\Qf$ be a biconservative isometric immersion with
nonzero parallel mean curvature vector field $H$.Then $f$ is 
either an open subset of a slice $\Qt^4_\epsilon\times\{t_0\}$
for some $t_0\in\R$, an open subset of a Riemannian product
$N^3\times\R$, where $N^3$ is a hypersurface of $\Qt^4_\epsilon$,
or it is locally congruent to the immersion $f$ described in Theorem \ref{theo:main1}. In particular, $f$ belongs to class $\A$.
\end{theorem}

\bibliographystyle{amsplain}

\Addresses

\end{document}